\documentclass[10pt,reqno]{amsart}
\usepackage[all]{xy}
\usepackage[parfill]{parskip}

\usepackage{verbatim}
\usepackage{color}

\usepackage{amssymb}
\usepackage{amsmath, amscd, graphicx, latexsym, hyperref, times}
\usepackage{graphicx}
\usepackage[abs]{overpic}
\usepackage{tikz}
\usepackage{tikz-cd}
\usetikzlibrary{arrows, patterns}

\oddsidemargin .25in \theoremstyle{plain}

\allowdisplaybreaks[3]

\def\Xint#1{\mathchoice
  {\XXint\displaystyle\textstyle{#1}}%
  {\XXint\textstyle\scriptstyle{#1}}%
  {\XXint\scriptstyle\scriptscriptstyle{#1}}%
  {\XXint\scriptscriptstyle\scriptscriptstyle{#1}}%
  \!\int}
\def\XXint#1#2#3{{\setbox0=\hbox{$#1{#2#3}{\int}$}
  \vcenter{\hbox{$#2#3$}}\kern-.5\wd0}}

\def\dashint{\Xint-}

\usepackage{cleveref}
\newtheorem{theorem}{Theorem}[section]
\newtheorem{lemma}[theorem]{Lemma}
\newtheorem{definition}[theorem]{Definition}
\newtheorem{proposition}[theorem]{Proposition}
\newtheorem{corollary}[theorem]{Corollary}
\newtheorem{conjecture}[theorem]{Conjecture}
\newtheorem{remark}[theorem]{Remark}
\newtheorem{claim}[theorem]{Claim}

\crefname{equation}{}{}
\crefname{section}{Section}{Section}
\crefname{figure}{Figure}{Figure}
\usepackage{appendix}
\begin{document}

\title[Minimal Degrees, Volume Growth, and Curvature Decay]
{Minimal Degrees, Volume Growth, and Curvature Decay on Complete K\"ahler Manifolds}
\author{Yuang Shi}
\address{School of Mathematical Sciences, Shanghai Jiao Tong University}
\email{yuangshi@sjtu.edu.cn}

\date{}
\begin{abstract}
We consider noncompact complete K\"ahler manifolds with nonnegative bisectional curvature. Our main results are: 1. Precise relations among refined minimal degree of polynomial growth holomorphic functions and holomorphic volume forms, $\operatorname{AVR}$ (asymptotic volume ratio) and $\operatorname{ASCD}$ (average of scalar curvature decay) are established. 2. The Lyapunov asymptotic behavior of the K\"ahler-Ricci flow can be described in terms of polynomial growth holomorphic functions. This provides a unifying perspective that bridges the two distinct proofs of Yau’s uniformization conjecture by Liu and Chau-Lee-Tam. These resolve two conjectures made by Yang.
\end{abstract}
\maketitle
\section{Introduction}\label{sec:intro}
As part of Yau's program to study complex manifolds of parabolic type, he proposed the following well-known longstanding uniformization conjecture in 1970s:

\begin{conjecture}[Yau's Uniformization Conjecture, \cite{Yau91}] \label{Yau conjecture}
    A complete noncompact K\"ahler manifold $(M^n,g)$ with positive bisectional curvature is biholomorphic to $\mathbb{C}^n$.
\end{conjecture}

The full conjecture is still widely open. Nevertheless, there has been much progress. In the maximal volume growth case, Liu’s breakthrough \cite{Liu19} confirmed the conjecture by combining Gromov-Hausdorff convergence techniques with the three-circle theorem developed in \cite{Liu16c}. We mention that the parabolic method of Ni-Tam \cite{NT03} plays a significant role in constructing strictly plurisubharmonic functions near the tangent cone. In fact, elliptic methods suffice, see \cite{Liu21}. An alternative proof, based on K\"ahler-Ricci flow, was later provided by Lee-Tam \cite{LT20}, building on the result of Chau-Tam \cite{CT06}.

 These advances lead to a deeper understanding of the structure of such manifolds. Notably, through the works of \cite{Ni04,Ni05,NT13,Liu16a,Liu16b,Liu21}, the following conjecture of Ni has been established as a theorem:
 \begin{theorem}[Corollary 3.2 in \cite{Ni04}, Theorem 1.2 in \cite{NT13}, Theorem 2 in \cite{Liu16b}, Theorem 1.4 in \cite{Liu16a}, Corollary 2.16 in \cite{Liu21}]\label{Niconj}
 Let $(M^n,g)$ be a complete noncompact K\"ahler manifold with nonnegative bisectional curvature. Assume that the universal cover of $M$ does not split. Then the following conditions are equivalent:

     (1) $M$ is of maximal volume growth, i.e.
     $$
     \operatorname{AVR}(M,g)=\lim_{r\to \infty}\frac{\operatorname{Vol}\left(B(p,r)\right)}{\omega_{2n}r^{2n}}>0.
     $$
     
     Here $\omega_{2n}$ is the volume of the unit ball in $\mathbb{C}^n$.

     (2) There exists a nonconstant polynomial growth holomorphic function, i.e. $\mathcal{O}_P(M)\neq \mathbb{C}$.

     (3) The average scalar curvature decay is finite, i.e.
     $$
     \operatorname{ASCD}(M,g)= \limsup\limits_{r\to\infty} r^2\dashint_{B(p, r)}S
     $$
     
     is finite. Here $S$ is the scalar curvature.
 \end{theorem}
\begin{remark}
    (1) Note that both $\operatorname{AVR}(M,g)$ and $\operatorname{ASCD}(M,g)$ are independent of the choice of $p$.

    (2) A very recent result of Liu \cite{Liu24} shows that for any $(M,g)$ with nonnegative bisectional curvature, the ``$\limsup$'' in the definition of $\operatorname{ASCD}(M,g)$ above can be replaced by ``$\lim$'', see also \cite{Ni12}.
\end{remark}

In his thesis \cite{Yang13}, Yang defined:
\begin{definition}
    Let $(M^n,g)$ be a complete K\"ahler manifold with nonnegative bisectional curvature, for a fixed $p\in M$, define
   \begin{align*}
d_{\operatorname{min}}
&=\inf_{f\in \mathcal{O}_P(M)}
\left\{
\operatorname{deg}(f)
=\limsup_{x\to\infty}\frac{\log |f(x)|}{\log d(x,p)}
\;\middle|\;
f \text{ is nonconstant }
\right\}.\\
D_{\operatorname{min}}
&=\inf_{s\in P(M,\mathcal{K}_M)}
\left\{
\operatorname{deg}(s)
=\limsup_{x\to\infty}\frac{\log \|s(x)\|}{\log d(x,p)}
\;\middle|\;
s \text{ is nonzero }
\right\}.
\end{align*}
where $\mathcal{K}_M$ is the canonical line bundle over $M$, $\mathcal{O}_P(M)$ is the space of all polynomial growth holomorphic functions on $M$ and $P(M,\mathcal{K}_M)$ is the space of all polynomial growth holomorphic $n$-forms on $M$. Set $d_{\operatorname{min}}=+\infty$ or $D_{\operatorname{min}}=+\infty$ if $M$ does not admit any nonconstant holomorphic functions of polynomial growth or $\mathcal{K}_M$ admits no nonzero holomorphic sections of polynomial growth.
\end{definition}

Now assume $(M^n,\omega)$ satisfies the assumption in Theorem \ref{Niconj} and is of maximal volume growth. Since the Kodaira dimension $K(M)=n$, $\mathcal{O}_P(M)$ is ``holomorphically regular'' in the sense that we can always find ``local coordinate by global polynomial growth holomorphic functions''. The refined minimal degree can be also defined in \cite{Yang13} by 
$$
\overrightarrow{d_{\operatorname{min}}}(p) := 
\left(d_{\operatorname{min}}^{(1)}, \cdots, d_{\operatorname{min}}^{(n)}\right) := 
\inf_{\{f_{1}, \cdots, f_{n}\}} 
\left\{ 
    \limsup_{x \rightarrow \infty} \frac{\log |f_{1}(x)|}{\log d(x, p)}, 
    \cdots, 
    \limsup_{x \rightarrow \infty} \frac{\log |f_{n}(x)|}{\log d(x, p)} 
\right\}
$$
where the infimum is taken among any $n$-tuple of global holomorphic functions that gives local coordinate at $p$ with the corresponding 
$
\limsup\limits_{x \rightarrow \infty} \frac{\log |f_{i}(x)|}{\log d(x, p)} 
$ 
arranged in a non-decreasing order for $1 \leq i \leq n$, we denote this set by $\mathcal{O}_{P}(M,p)$. In other words, for any $1\leq k\leq n$ we have $d_{\operatorname{min}}^{(k)} = \inf\limits_{f_{k}} \limsup\limits_{x\rightarrow\infty} \frac{\log |f_{k}(x)|}{\log d(x,p)}$, where the infimum is taken among all possible $f_{k}$ that appears in the $k$-th component of some sequence in $\mathcal{O}_{P}(M,p)$. Note that a priori it's unclear if $\left(d_{\operatorname{min}}^{(1)}, \cdots, d_{\operatorname{min}}^{(n)}\right)$ can be obtained by an n-tuple of holomorphic functions in $\mathcal{O}_{P}(M,p)$. But obviously $d_{\operatorname{min}}=d_{\operatorname{min}}^{(1)}$.

Then Yang formulated the following conjecture on the relation between the above quantities, which can be understood as the quantitative version of Theorem \ref{Niconj}.
\begin{conjecture}[Conjecture 2.5.6, Conjecture 2.5.8 in \cite{Yang13}]\label{YConj1}
    Let $(M^n,g)$ be a complete noncompact K\"ahler manifold with nonnegative bisectional curvature. Assume the universal cover of $M$ does not split. Then

    (1) If $M$ is of maximal volume growth, then $\overrightarrow{d_{\operatorname{min}}}(p)$ can be realized by an n-tuple of holomorphic functions in $\mathcal{O}_{P}(M,p)$ and is independent of the choice of $p$.

    (2) $\operatorname{AVR}(M,g)=\prod\limits_{i=1}^n\frac{1}{d_{\operatorname{min}}^{(i)}}$.

    (3) $D_{\operatorname{min}}=\sum\limits_{i=1}^n d_{\operatorname{min}}^{(i)}-n$.

    (4) $\operatorname{ASCD}(M,g)=4n D_{\operatorname{min}}$.
\end{conjecture}
In the case of nonmaximal volume growth, from Theorem \ref{Niconj}, $d_{\min}=+\infty$ and $\operatorname{ASCD}=+\infty$. So (2) is true.

Assume $M$ is of maximal volume growth, in \cite{Yang13}, Yang verified that the conjecture holds for $U(n)$-invariant K\"ahler metrics on $\mathbb{C}^n$. Liu proved (1) implicitly in \cite{Liu19}. In fact, he showed that any n-tuple of polynomial growth functions which are algebraically independent and of minimal degrees can form a global coordinate. Then recently in \cite{Liu24}, Liu proved (2) and derived the explicit formula $\operatorname{ASCD}(M,g)=4n( \sum\limits_{i=1}^n d_{\operatorname{min}}^{(i)}-n)$ by passing the geometric quantities to the tangent cones and exploiting the metric K\"ahler cone structure of the limit space. Finally in \cite{Yang22}, Yang confirmed (3) for expanding gradient K\"ahler-Ricci solitons with nonnegative Ricci curvature by analyzing the Poincar\'e coordinate introduced in \cite{Bry08}, see Theorem \ref{Consoli}. 

Therefore, the general proof of (3) remained the final step toward resolving Conjecture \ref{YConj1} in full generality.

In this paper, we demonstrate part (3) of Conjecture \ref{YConj1}, thereby completely resolving the conjecture. Precisely,
\begin{theorem}\label{thm:A}
    Let $(M^n,g)$ be a complete noncompact K\"ahler manifold with nonnegative bisectional curvature. Suppose the universal cover of $M$ does not split. Then
    $$
    D_{\operatorname{min}}=\sum\limits_{i=1}^n d_{\operatorname{min}}^{(i)}-n.
    $$
\end{theorem}
Now we move to the behavior of the K\"ahler-Ricci flow on complete noncompact K\"ahler manifolds with nonnegative bisectional curvature.
\begin{equation}\label{KRflow}
\frac{\partial g(t)}{\partial t}=-\operatorname{Ric}(g(t)),\quad g(0)=g.
\end{equation}

When the initial metric $(M,g)$ has maximal volume growth, the long-time behavior of (\ref{KRflow}) has been described very clearly. Let us summarize these results in the following:
\begin{theorem}[Corollary 1 in \cite{Ni05}, Theorem 1.2, Proposition 3.2 in \cite{CT06}, Theorem 6.1 in \cite{CT11}, Theorem 1.5 in \cite{LT20}, Theorem 1.1 in \cite{Lee25}]\label{LongKRflow}
    Suppose that $(M^n,J,g)$ is a complete noncompact K\"ahler manifold with nonnegative bisectional curvature and maximal volume growth such that $\operatorname{AVR}(M,g)=\nu>0$. Then the following conclusions hold:
    
    (1) There exists a unique long-time solution $g(t)$ to the K\"ahler-Ricci flow on $M\times [0,+\infty)$ with the initial metric $g$ such that,
    \begin{enumerate}
            \item[(i)] It has nonnegative bisectional curvature for any $t\geq 0$.
            \item[(ii)] There exists a constant $C(n,\nu)$ such that 
          $\|\operatorname{Rm}(x,t)\|\leq\frac{C(n,\nu)}{t}$ for any $x\in M$ and $t\in(0,+\infty)$.
          \item[(iii)] $\operatorname{AVR}(M,g(t))=\operatorname{AVR}(M,g)=\nu$.
    \end{enumerate}
    
    (2) For any point $x\in M$, let $\{\lambda_{1}(x,t),\cdots,\lambda_{n}(x,t)\}$ denote the 
          eigenvalues of Ricci curvature $\operatorname{Ric}(x,t)$ with respect to $g(t)$ in the 
          nondecreasing order. Then $t\lambda_{i}(x,t)$ is nondecreasing on $t>0$, hence 
          $\mu_{i}(x) := \lim\limits_{t\rightarrow+\infty}t\lambda_{i}(x,t)$ exists.

    (3) If $\mu_1(x)<\mu_2(x)<\cdots<\mu_l(x)$ are the distinct limits in (2), where $l\leq n$. Then $V=T_x^{(1,0)}(M)$ can be decomposed orthogonally with respect to $g$ as $V_1\oplus \cdots\oplus V_l$ so that 

    if $v$ is a nonzero vector in $V_i$ for some $1\leq i\leq l$, and let $v(t)=\frac{v}{\|v\|_{g(t)}}$, then
    $$
    \lim\limits_{t\to\infty}t\operatorname{Ric}(v(t),\overline{v(t)})=\mu_i
    $$
    and thus
    $$
    \lim\limits_{t\to\infty}\frac{\log \|v\|_{g(t)}^2}{\log t}=-\mu_i.
    $$
    Moreover, both convergences are uniform over all $v\in V_i\backslash \{0\}$.
    
    (4)
    \begin{equation}\label{volformchange}
    -\sum\limits_{i=1}^l \mu_i(x)\operatorname{dim}_{\mathbb{C}}V_i=\lim\limits_{t\to+\infty}\frac{\log \det (g_{i\bar{j}}(x,t))}{\log t}.
    \end{equation}

    (5) Fix any point $p\in M$. Given any $t_{k}\rightarrow+\infty$, define $g_{k}(t) = \frac{1}{t_{k}}g(t_{k}t)$. 
          The pointed sequence $(M^{n},J,g_{k}(t),p)$ sub-sequentially converges to an expanding gradient K\"ahler-Ricci soliton $(N,J_{\infty},h(t),O)$ where $t\in(0,+\infty)$ in the following sense:
          \begin{enumerate}
            \item[(i)] After picking a subsequence still denoted by $g_{k}$, there exists an increasing 
                  sequence of open subsets $O\in U_{k}$, which exhausts $N$ and a family of 
                  diffeomorphisms $F_{k}:U_{k}\rightarrow F_{k}(U_{k})\subset M$ with $F_{k}(O)=p$.
            \item[(ii)] As $t_{k}\rightarrow\infty$, the sequence $(U_{k},F_{k}^{*}J,F_{k}^{*}(g_k(t)),p)$ 
                  converges smoothly to another sequence of complete K\"ahler manifolds 
                  $(N,J_{\infty},h(t),O)$ uniformly on compact sets of $N\times(0,\infty)$.
            \item[(iii)] $(N,J_{\infty},h(t))$ has nonnegative bisectional curvature for any $t>0$, and 
                  there exists a real-valued function $f\in C^{\infty}(N)$ such that 
                  $(N,J_{\infty},h(1))$ satisfies the expanding soliton equation
                  \begin{equation}\label{eq:soliton}
                    \operatorname{Ric}_{i\bar{j}}(h(1)) + h_{i\bar{j}}(1) - f_{i\bar{j}} = 0, \quad f_{ij} = f_{\bar{i}\bar{j}} = 0.
                  \end{equation}
                  Moreover, $\nabla_{h(1)}f(O)=0$ and the eigenvalues of Ricci curvature of $h(1)$ at $O$, 
                  arranged in the non-decreasing order, equal $\mu_{i}(p)$ for $1\leq i\leq n$.
          \end{enumerate}
\end{theorem}
\begin{remark}
    (1) If $(M,g)$ has bounded curvature, the above theorem was proved by Chau–Tam in \cite{CT06, CT11} through Li-Yau-Hamilton estimate.  
In the case of unbounded curvature, the existence of such solution was established by Lee–Tam in \cite{LT20, LT21}, in which they constructed a pyramid solution to the K\"ahler–Ricci flow.  
Moreover, the uniqueness of complete non-compact Ricci flows under a scaling-invariant curvature bound was proved very recently by Lee \cite{Lee25}.

    (2) In the context of the K\"ahler–Ricci flow, the theorem asserts that $(M,J,g(t))$ asymptotically approaches an expanding gradient K\"ahler–Ricci soliton at the base point $p$. Moreover, conclusions (2) and (3) essentially indicate that the Ricci curvature $\operatorname{Ric}(x,t)$ becomes ``simultaneously diagonalizable'' as $t\rightarrow +\infty$, in an appropriate asymptotic sense. In fact, $\mu_i$'s play the role of Lyapunov exponents in a dynamical-system interpretation on asymptotics of (\ref{KRflow}).
\end{remark}

Then Yang proposed the following conjecture in \cite{Yang13}:
\begin{conjecture}[Conjecture 2.5.16 in \cite{Yang13}]\label{YConj2}
    Let $(M^n,g)$ be a complete noncompact K\"ahler manifold with nonnegative bisectional curvature and maximal volume growth.  Let $g(t)$, $t\in [0,+\infty)$, be the unique complete solution K\"ahler-Ricci flow with initial metric $g$ in Theorem \ref{LongKRflow}, then

    (1) $\mu_i(p)$ is independent of the choice of $p$ for any $1\leq i\leq n$.

    (2) $\mu_i=d_{\operatorname{min}}^{(i)}-1$ for any $1\leq i\leq n$.

    (3) $\operatorname{ASCD}(M,g(t))$ is invariant along $g(t)$.
\end{conjecture}

In \cite{Yang22}, Yang proved (1) using Shi's curvature estimates on Ricci flow solutions under a scaling-invariant curvature bound. He further proved (2) for expanding K\"ahler-Ricci solitons with nonnegative Ricci curvature. Precisely,

\begin{theorem}[Theorem 1.4 in \cite{Yang22}]\label{Consoli}
    Let $(N^n,J,O,g,f)$ be a complete noncompact expanding gradient K\"ahler-Ricci soliton with nonnegative Ricci curvature, normalized so that $f=R+|\nabla^gf|^2$. Let $0\leq\mu_1\leq\cdots\leq \mu_n$ be the eigenvalues of Ricci curvature at $O$. Then
       \begin{equation}\label{minisoli}
       \overrightarrow{d_{\operatorname{min}}}(q)=\{\mu_1+1,\cdots,\mu_n+1\}\quad \text{for all } q\in N;
       \end{equation}
       and
        \begin{equation}\label{minisoliform}
        D_{\operatorname{min}}=\sum\limits_{i=1}^n\mu_i.
        \end{equation}
\end{theorem}

In this paper, our second main result is that Conjecture \ref{YConj2} holds in general.
\begin{theorem}\label{thm:B}
    Conjecture \ref{YConj2} holds.
\end{theorem}

Liu \cite{Liu19} used the three-circle theorem to show that the degrees of polynomial growth functions can also be ``simultaneously diagonalized'' along the metric blow-down sequence, we will discuss the phenomenon in Theorem \ref{infinitystrucutre} (5) in Section \ref{sec:2}. As a by-product, we can describe the Lyapunov regularity in (3) of Theorem \ref{LongKRflow} using the polynomial ring $\mathcal{O}_P(M)$. In some sense, it bridges the two distinct proofs of Yau’s uniformization conjecture in the case of maximal volume growth.
\begin{corollary}\label{bridge}
    In part (3) of Theorem \ref{LongKRflow}, the basis of $V_i$ can be taken as $\left\{\frac{\partial}{\partial f_{i1}}(x),\cdots,\frac{\partial}{\partial f_{im_i}}(x)\right\}$. Here $\operatorname{dim}V_i=m_i$ is the multiplicity of $\mu_i$, $\sum\limits_{i=1}^lm_i=n$, and
    $$
    f_{is}\in \mathcal{O}_P(M),\quad\operatorname{deg} (f_{is})=d_{\operatorname{min}}^{(i)}
    $$
    for any $1\leq s\leq m_i$.

    Moreover, $\{f_{11},\cdots,f_{1m_1};\cdots;f_{l1},\cdots,f_{lm_l}\}$ can serve as a biholomorphism from $M$ onto $\mathbb{C}^n$.
\end{corollary}

This paper is organized as follows. Section \ref{sec:2} gives some basic preliminary results and some simple conclusions that will be used later. In Section \ref{sec:3}, we will prove Theorem \ref{thm:B} by noting that any tangent cone of $M$ also arises as a tangent cone of corresponding expanding K\"ahler-Ricci solitons described in Theorem \ref{LongKRflow}, we then deduce Corollary \ref{bridge}. Section \ref{sec:4} contains the proof of Theorem \ref{thm:A}. In Section \ref{sec:5} we discuss several related problems.
\section{Preliminary results}\label{sec:2}
\subsection{Structure of K\"ahler manifolds with maximal volume growth and tangent cones}  We first collect several results on the structure of complete noncompact K\"ahler manifolds with nonnegative bisectional curvature and maximal volume growth and their tangent cones, and prove some relevant facts that will be used later.

The following is the main theorem in \cite{Liu19}.
\begin{theorem}\label{Liuthm1}
    Let $(M^n,g)$ be a complete noncompact K\"ahler manifold with nonnegative bisectional curvature and maximal volume growth, $p\in M$. Let $\operatorname{AVR}(M,g)=\nu>0$.
    Then
    
    (1) $M$ is biholomorphic to $\mathbb{C}^n$.

    (2) There exists an n-tuple of polynomial growth holomorphic functions $(f_1,\cdots,f_n)$ serving as a proper biholomorphism to $\mathbb{C}^n$. These functions satisfy
    \begin{equation}\label{normalization}
\dashint_{B(p,1)}f_i\overline{f_j}=\delta_{ij},\quad f_{i}(p)=0\quad  \forall i,j\in \{1,\cdots,n\}.
    \end{equation}
    Moreover, any element in $\mathcal{O}_P(M)$ is a polynomial of these functions, meaning that $\mathcal{O}_P(M)\cong \mathbb{C}[f_1,\cdots,f_n]$.

    (3) Any two such global coordinates (consisting of n polynomial growth holomorphic functions satisfying (\ref{normalization}), serving as a proper biholomorphism onto $\mathbb{C}^n$, and generating $\mathcal{O}_P(M)$) differ by a constant orthogonal transformation. 
\end{theorem}
\begin{proof}
    The proof of these conclusions is implicitly contained in \cite{Liu16a,Liu19}. For reader's convenience, we give a thorough explanation. The proof of (1) and (2) is contained in Section 4 in \cite{Liu19}.
    
    The properness is also contained in \cite{Liu16a}. It's shown that there exists a constant $D=D(n,\nu)$ such that let $(g_1,\cdots, g_k)$ be a linearly independent basis of $\mathcal{O}_D(M)$, where $k=\operatorname{dim}_{\mathbb{C}}\mathcal{O}_D(M)-1$, such that
    $$
    \dashint_{B(p,1)}g_i\overline{g_j}=\delta_{ij},\quad g_{i}(p)=0\quad \forall i,j\in \{1,\cdots,k\}.
    $$
    and
    $$
    \min\limits_{\partial B(p,r)}\sum\limits_{i=1}^k|g_i|^2\geq c(n,\nu)r^2
    $$
    holds for any $r>0$ and some constant $c=c(n,\nu)$. Moreover, $(g_1,\cdots,g_k)$ serves as an embedding of $M$ onto an affine variety in $\mathbb{C}^{k}$.

    Take an $n$-tuple of algebraically independent holomorphic functions with minimal degrees in $(g_1,\cdots,g_k)$, denoted by $(f_1,\cdots,f_n)$. Following the proof in Section 4 in \cite{Liu19}, $(f_1,\cdots,f_n)$ serves as a biholomorphism of $M$ onto $\mathbb{C}^n$, and any polynomial growth holomorphic function is a polynomial of them.

    Since any $g_i$ is a polynomial of these $(f_1,\cdots,f_n)$ with degree at most $D$ according to the three-circle theorem, there exists a constant $C$ such that
    $$
     c(n,\nu)r^2\leq \min\limits_{\partial B(p,r)}\sum\limits_{i=1}^k|g_i|^2\leq C\left(\min\limits_{\partial B(p,r)}\sum\limits_{i=1}^n|f_i|^2\right)^D\quad \text{for any } r>0,
    $$
    from which we can deduce
    \begin{equation}\label{propercoord}
    \min\limits_{\partial B(p,r)}\sum\limits_{i=1}^n|f_i|^2\geq c(n,\nu)r^{\frac{2}{D}}
    \end{equation}
    for any $r>0$ and some constant $c=c(n,\nu)$. In other words, $(f_1,\cdots,f_n)$ is proper.

    For (3), given any two such global coordinates $(f_1,\cdots,f_n)$ and $(h_1,\cdots,h_n)$. Assume their degrees are both arranged in the nondecreasing order, then $\left(\frac{\partial h_i}{\partial f_j}\right)$ and $\left(\frac{\partial f_i}{\partial h_j}\right)$ are both polynomial matrices of $(f_1,\cdots,f_n)$ and are inverse of each other. Therefore they must both be constant matrices. Then due to the normalization condition (\ref{normalization}), these two matrices are both orthogonal. 
\end{proof}
In fact, the minimal degrees of the coordinate in Theorem \ref{Liuthm1} (3) described above are just the refined minimal order $\overrightarrow{d_{\operatorname{min}}}$ defined in the Introduction, and this has been shown in \cite{Liu19}.
   \begin{corollary}\label{refmini}
       Conjecture \ref{YConj1} (1) holds.
   \end{corollary}
   \begin{proof}
       If $M$ is of nonmaximal volume growth. By Theorem 2 in \cite{Liu16b}, $\mathcal{O}_{P}(M)=\mathbb{C}$, i.e. $\overrightarrow{d_{\operatorname{min}}}=\overrightarrow{+\infty}$, which is independent of the choice of $p$.

       If $M$ is of maximal volume growth, due to \cite{Liu19}, there exists a strictly increasing sequence $1\leq d_1<d_2<\cdots$, so that $\mathcal{O}_P(M)=\mathbb{C}\oplus \left(\mathcal{O}_{d_1}(M)/\mathbb{C}\right)\oplus\left(\mathcal{O}_{d_2}(M)/\mathcal{O}_{d_1}(M)\right)\oplus\cdots$ as a complex vector space. For any $k\in \mathbb{N}$, pick a maximal linearly independent vectors $f_{k1},\cdots,f_{km_k}$ of $\mathcal{O}_{d_k}(M)$ so that they form a basis of $\mathcal{O}_{d_k}(M)/\mathcal{O}_{d_{k-1}}(M)$ as quotient of vector spaces and no element in the span is given by polynomials of $\mathcal{O}_{d_{k-1}}(M)$. Then the first $n$ functions $\{f_{11},\cdots,f_{1m_1},f_{21},\cdots\}$ form a global coordinate in Theorem \ref{Liuthm1} after normalization as (\ref{normalization}).

       By the definition of the refined minimal degree, $\overrightarrow{d_{\operatorname{min}}}$ is just the n-tuple of degrees of this global coordinate $\{f_{11},\cdots,f_{1m_1},f_{21},\cdots\}$, which depends only on $M$ itself by (3) in Theorem \ref{Liuthm1}.
   \end{proof}
   \begin{remark}
       (1) As can be seen in the proof, if $M$ is of maximal volume growth, Corollary \ref{refmini} still holds without the nonsplitting condition on the universal cover.

       (2) For simplicity, if $M$ is of maximal volume growth, we just call the coordinate as in (2) in Theorem \ref{Liuthm1} a \textbf{canonical coordinate} on $M$, with degrees $\left(d_{\operatorname{min}}^{(1)}, \cdots, d_{\operatorname{min}}^{(n)}\right)$. And we arrange them in non-decreasing order of their degrees.
   \end{remark}

Now we recall the definition of ``$\operatorname{BK}\ge 0$'' in Lott \cite{Lott21}, for the purpose we are using here, we will only consider the case of complex manifolds. 
\begin{definition}\label{lott}
(1) Let $X$ be a complex manifold. Let $\{U_i\}_{i\in I}$ be an open covering of $X$, if on each $U_i$ there exists a continuous plurisubharmonic function $\phi_i$, so that $\phi_i-\phi_j$ is pluriharmonic on each $U_i\cap U_j\neq \varnothing$. Then $\omega=\sqrt{-1}\partial\bar{\partial}\phi_i$ on $U_i$ is a well-defined positive (1,1)-current, and we will call it a K\"ahler current on $X$.

(2) Let $d$ be a metric on such $(X,\omega)$, we say that the metric $d$ is compatible with the K\"ahler current $\omega$, if the metric $d$ induces the topology of $X$, so that if $\Sigma$ is an embedded holomorphic disk, then $\omega|_{\Sigma}$ equals the 2-dimensional Hausdorff measure induced by $d$ on each $\Sigma \cap U_i \neq \varnothing$.

(3) Given any triple $(X,\omega,d)$ with $X$ a complex manifold, $\omega$ a K\"ahler current, $d$ a metric compatible with $\omega$. We call $X$ has ``$\operatorname{BK}\ge 0$'' if for any $i\in I$, $\phi_i-d_p^2/2$ is plurisubharmonic on $U_i$ for any $p\in X$, here $d_p$ is the distance function from $p$. In fact, if $(X,\omega,d)$ is a smooth K\"ahler manifold, the condition ``$\operatorname{BK}\ge 0$'' is equivalent to that bisectional curvature is nonnegative.
\end{definition}

The structure of tangent cones of infinity of $M$ has been studied in \cite{Liu18a,Liu18b,LT21,Lott21,Liu24}. Let's summarize these results in the following:

   \begin{theorem}\label{infinitystrucutre}
       Let $(M^n,g)$ be a complete noncompact K\"ahler manifold with nonnegative bisectional curvature and maximal volume growth. Let $(M_{\infty},p_{\infty},d_{\infty})$ be a tangent cone at infinity of $M$, then:
       
       (1) $(M_{\infty},p_{\infty},d_{\infty})$ is a complex manifold which is biholomorphic to $\mathbb{C}^n$. Moreover, it has "$BK\geq 0$" with K\"ahler current $\omega=\frac{\sqrt{-1}}{2}\partial\bar{\partial}r^2$, where $r$ is the radial function on $M_{\infty}$.

       (2) There exists a complete family of expanding gradient K\"ahler-Ricci solitons coming out of $M_{\infty}$, i.e. there exists $(M_{\infty},h(t))$ a family of expanding gradient K\"ahler-Ricci soliton metrics with nonnegative bisectional curvature such that $h(t)\rightarrow d_{\infty}$ as $t\rightarrow 0$ in the Gromov-Hausdorff sense.
       
       (3) $\operatorname{dim}_{\mathbb{C}}\mathcal{O}_{d}(M_{\infty})=\operatorname{dim}_{\mathbb{C}}\mathcal{O}_{d}(M)$ for any $d\geq 0$. And there exists a n-tuple of polynomial growth functions $(z_{1\infty},\cdots,z_{n\infty})$ serving as a proper biholomorphism onto $\mathbb{C}^n$. Each $z_{i\infty}$ is of homogeneous degree $d_{\operatorname{min}}^{(i)}$. They satisfy
    \begin{equation}\label{normalizationcone}
\dashint_{B(p_{\infty},1)}z_{i\infty}\overline{z_{j\infty}}=\delta_{ij},\quad z_{i\infty}(p_{\infty})=0\quad \forall i,j\in \{1,\cdots,n\}.
    \end{equation}
    Moreover, any element in $\mathcal{O}_P(M_{\infty})$ is a polynomial of these functions. In particular, $\mathcal{O}_P(M_{\infty})\cong \mathbb{C}[z_{1\infty},\cdots,z_{n\infty}]$.

    (4) Any two such global coordinates (consisting of n polynomial growth homothetically homogeneous, holomorphic functions satisfying (\ref{normalizationcone}), serving as a proper biholomorphism onto $\mathbb{C}^n$, and generating $\mathcal{O}_P(M_{\infty})$) differ by a constant orthogonal transformation.

    (5) Suppose there is a blow-down sequence $(M_i,g_i,p_i)=(M,r_i^{-2}g,p)$ that converges to $(M_\infty,d_{\infty},p_\infty)$ in the Gromov-Hausdorff sense. Let $(z_1,\cdots,z_n)$ be a canonical coordinate on $M$. Normalize them to $(z_{1k},\cdots,z_{nk})$ on $M_k$ such that
    \begin{equation}\label{normalizationsequence}
     \dashint_{B(p_{k},1)}z_{ik}\overline{z_{jk}}=\delta_{ij},\quad z_{ik}(p_{k})=0\quad \forall i,j\in \{1,\cdots,n\}, k\in \mathbb{N}.
    \end{equation}
    Then $(z_{1k},\cdots,z_{1n})$ sub-sequentially converges to a global coordinate $(z_{1\infty},\cdots,z_{n\infty})$ as in (3) on $M_{\infty}$ uniformly on compact subsets as $k\rightarrow +\infty$. 

    (6) The volume of the unit ball in $M_{\infty}$ is $\frac{\omega_{2n}}{d_{\operatorname{min}}^{(1)}\cdots d_{\operatorname{min}}^{(n)}}$. In particular, Conjecture \ref{YConj1} (2) holds.
   \end{theorem}
   \begin{proof}
       (1) and (2) are proved in Proposition 6.1 in \cite{Lott21}. The coming out of a family of expanding gradient K\"ahler-Ricci solitons follows from Proposition 3.2 of Chau-Tam \cite{CT06} using Cao’s harnack inequality.

       For (3), for any $d\geq 0$, $\operatorname{dim}_{\mathbb{C}}\mathcal{O}_{d}(M_{\infty})=\operatorname{dim}_{\mathbb{C}}\mathcal{O}_{d}(M)$ by Proposition 4.1 in \cite{Liu19}. Next we show the existence of such coordinate.

       Firstly we claim that the homothetic vector field $r\frac{\partial}{\partial r}$ is holomorphic on $M_{\infty}$. By Theorem \ref{LongKRflow} (1) and (5), there exists a long-time K\"ahler-Ricci flow solution $g(t)$ on $M$ such that $(M_\infty,J_{\infty},p_{\infty},h(t))$ is the pointed Cheeger-Hamilton limit of a blow-down sequence of $g(t)$, i.e. $(M,J,p,g_i(t)=\frac{1}{t_i}g(t_it))$ for some sequence $t_i\rightarrow +\infty$. In other words, $J$ converges to $J_{\infty}$ through a sequence of diffeomorphisms exhausted on $M_{\infty}$. By Claim 4.3 in \cite{Liu19}, the complexification of $r\frac{\partial}{\partial r}$ is in the span of several holomorphic vector fields on $M_\infty$. In particular, $r\frac{\partial}{\partial r}$ is smooth and real holomorphic. Moreover, $(r\frac{\partial}{\partial r})\mathcal{O}_d(M_\infty)\subset \mathcal{O}_d(M_\infty)$ holds for any $d\geq 0$. The claim is confirmed.

       Therefore, $r\frac{\partial}{\partial r}$ is a contracted holomorphic field on $M_\infty$, from a result in \cite{RR88}, $M_\infty$ is biholomorphic to $\mathbb{C}^n$. Now following a similar argument to that in Section 4 of \cite{Liu19}, replacing the holomorphic vector field $X$ in the proof by $r\frac{\partial}{\partial r}$, we conclude the proof of (3) except the homogenity of these coordinate functions. This is also obvious since we can just choose the highest order terms of each coordinate functions and they certainly form a new biholomorphism onto $\mathbb{C}^n$. In fact, as in the proof of Corollary \ref{refmini}, take any $n$-tuple of algebraically independent homogeneous functions in $\mathcal{O}_P(M_{\infty})$ of minimal degree, they will serve as a biholomorphism onto $\mathbb{C}^n$. Finally we can renomalize them such that (\ref{normalizationcone}) is satisfied. Following the same argument of the proof of (3) in Theorem \ref{Liuthm1}, we get that any two such coordinates has same growth degrees, thus confirms (4). We will show that these coordinates have same degree $(d_{\operatorname{min}}^{(1)},\cdots,d_{\operatorname{min}}^{(n)})$ in (5).
       
       For (5), the mean value inequality \cite{LS84} implies that, 
       $$
       M_{z_{ik}}(\frac{1}{2})=\sup\limits_{B(p_k,\frac{1}{2})}|z_{ik}|\leq C(n),
       $$
       holds for any $1\leq i\leq n$ and $k\in \mathbb{N}$.
       By the three-circle theorem in \cite{Liu16c} and Cheng-Yau gradient estimate \cite{CY75},
       $$
       M_{z_{ik}}(r)\leq C(n)r^{d_{\operatorname{min}}^{(i)}},\quad |dz_{ik}|\leq C(n)r^{d_{\operatorname{min}}^{(i)}-1}
       $$
       on $B_{p_i}(r)$ for any $1\leq i\leq n$, $k\in \mathbb{N}$ and $r\geq \frac{1}{2}$. Arzela-Ascoli theorem implies that $z_{ik}$ converges to a holomorphic function $z_{i\infty}$ of homogeneous degree $d_{\operatorname{min}}^{(i)}$ uniformly on compact subsets. By (\ref{propercoord}) and the Gromov-Hausdorff convergence,
       \begin{equation}\label{propercone}
    \min\limits_{\partial B(p_{\infty},r)}\sum\limits_{i=1}^n |z_{i\infty}|^2\geq c(n,\nu)r^{\frac{2}{D(n,\nu)}}>0.
    \end{equation}
    Now we show these functions are algebraically independent. Suppose not, then there exists a nontrivial polynomial $P\in \mathbb{C}[x_1,\cdots,x_n]$ such that
    $$
    P(z_{1\infty},\cdots,z_{n\infty})\equiv 0
    $$
    on $M_{\infty}$. Lift to $M_k$ for large $k$ and use the properness (\ref{normalizationsequence}) of $(z_{1k},\cdots,z_{nk})$, it implies $|P(z_{1k},\cdots,z_{nk})|$ can be arbitrarily small on arbitrarily large ball in $\mathbb{C}^n$. Thus $P\equiv 0$. Therefore $(z_{1\infty},\cdots,z_{n\infty})$ must form a global coordinate as in (3), and they have degree $(d_{\operatorname{min}}^{(1)},\cdots,d_{\operatorname{min}}^{(n)})$.

    For (6), the first statement is confirmed in the proof of Theorem 3 in page 13 in \cite{Liu24}. This concludes the proof of Conjecture \ref{YConj1} (2) in the maximal volume growth case.
   \end{proof}
   \begin{remark}
       (1) As can be seen in the proof, if $M$ is of maximal volume growth, Conjecture \ref{YConj1} (2) still holds without the nonsplitting condition on the universal cover.

       (2) For simplicity, we just call the coordinate as in (2) in Theorem \ref{infinitystrucutre} a \textbf{canonical coordinate} on $M_{\infty}$, with degrees $\left(d_{\operatorname{min}}^{(1)}, \cdots, d_{\operatorname{min}}^{(n)}\right)$. And we arrange them in non-decreasing order of their degrees.
   \end{remark}

   We have the following homogeneity property of $\mathcal{O}_P(M)$. In particular, $\mathcal{O}_P(M)$ and $\mathcal{O}_P(M_{\infty})$ share the same graded algebraic structure.

\begin{lemma}\label{homogeneity}
Let $(M^n,g)$ be a complete noncompact K\"ahler manifold with nonnegative bisectional curvature and maximal volume growth. Then for any $f,g\in \mathcal{O}_P(M)$, 
$$
\operatorname{deg}(fg)=\operatorname{deg}(f)+\operatorname{deg}(g).
$$
\end{lemma}
\begin{proof}
   Take a canonical coordinate $(f_1,\cdots,f_n)$ on $M$ with 
   $$
   f_i(p)=0\quad \forall i\in\{1,\cdots,n\}.
   $$
   For any sequence $r_k\rightarrow +\infty$, on $(M_k,p_k)=(r^{-2}_kM,p)$, we can normalize them such that
   $$
   f_{ik}=\frac{f_i}{M_{ik}},\quad \text{ where } M_{ik}=\sup\limits_{B_g(p,r_k)}|f_i|
   $$
   for any $i\in\{1,\cdots,n\}$ and $k\in\mathbb{N}$. Therefore on each $M_k$,
   $$
   \sup\limits_{B(p_k,1)}|f_{ik}|=1.
   $$
   Suppose $(M_i,p_i)\rightarrow (M_\infty,p_\infty)$ in the sense of Gromov-Hausdorff convergence, as in Proposition 5 in \cite{Liu18b}, each $f_{ik}\rightarrow f_{i\infty}$ with $f_{i\infty}$ is of homoegeneous degree $d_{\operatorname{min}}^{(i)}$.

   As these functions generate $\mathcal{O}_P(M)$, we just need to prove: for any $f=f_1^{k_1}\cdots f_n^{k_n}$, $\operatorname{deg}(f)=\sum\limits_{i=1}^nk_id_{\operatorname{min}}^{(i)}$.
   
   Now let $f=f_1^{k_1}\cdots f_n^{k_n}$, by the three-circle theorem, for any $a>1$
   $$
   \begin{aligned}
       \lim_{r\to\infty}\frac{M_{f}(ar)}{M_f(r)}=a^{\operatorname{deg}(f)}.
   \end{aligned}
   $$
   On the other hand,
   $$
   \begin{aligned}
   \lim_{r\to\infty}\frac{M_{f}(ar)}{M_f(r)}&=\lim\limits_{r\to\infty}\frac{M_{f_1^{k_1}\cdots f_n^{k_n}}(ar)}{M_{f_1^{k_1}\cdots f_n^{k_n}}(r)}\\
   &=\lim\limits_{k\to\infty}\frac{M_{f_1^{k_1}\cdots f_n^{k_n}}(ar_k)}{M_{f_1^{k_1}\cdots f_n^{k_n}}(r_k)}\\
   &=\lim\limits_{k\to\infty}\frac{M_{f_{1k}^{k_1}\cdots f_{nk}^{k_n}}(a)}{M_{f_{1k}^{k_1}\cdots f_{nk}^{k_n}}(1)}\\
   &=\frac{M_{f_{1\infty}^{k_1}\cdots f_{n\infty}^{k_n}}(a)}{M_{f_{1\infty}^{k_1}\cdots f_{n\infty}^{k_n}}(1)}\\
   &=a^{\sum\limits_{i=1}^nk_id_{\operatorname{min}}^{(i)}}.
   \end{aligned}
   $$

   Comparing these we complete the proof.
\end{proof}

Recently, Chu-Hao \cite{CH25} obtained an optimal rigidity of the dimension estimate for polynomial growth holomorphic functions. Their proof is elementary. Now we give an alternative proof by making use of Liu's deep result in \cite{Liu16a}, which can greatly simplify their proof.

\begin{theorem}[Theorem 1.6 in \cite{CH25}]
    Let $(M^n,g)$ be a complete noncompact K\"ahler manifold with nonnegative bisectional curvature. Suppose there exists $d\geq 1$ such that
    $$
\operatorname{dim}_{\mathbb{C}}\mathcal{O}_d(M)>\operatorname{dim}_{\mathbb{C}}\mathcal{O}_d(\mathbb{C}^n)-\binom{n+d-2}{d-1},
    $$
    then $M$ is biholomorphically isometric to $\mathbb{C}^n$ with Euclidean metric.
\end{theorem}
\begin{proof}
    Suppose $M$ is of maximal volume growth, then take $(f_1,\cdots,f_n)$ be a canonical coordinate with minimal degree $(d^{(1)}_{\operatorname{min}},\cdots,d^{(n)}_{\operatorname{min}})$ with $d^{(i)}_{\operatorname{min}}\geq 1$ in non-decreasing order.

    From Lemma \ref{homogeneity}, $\operatorname{dim}_{\mathbb{C}}\mathcal{O}_d(M)=\#\{(m_1,\cdots,m_n)\in \mathbb{N}^n\mid \sum\limits_{i=1}^n m_id^{(i)}_{\min}\leq d\}$. We know that $M$ is biholomorphically isometric to $\mathbb{C}^n$ if and only if $d^{(i)}_{\min}=1$ for all $i$. Suppose $M$ is not isometric to $\mathbb{C}^n$. Then $d^{(n)}_{\min}>1$, then 
    $$
    \begin{aligned}
\operatorname{dim}_{\mathbb{C}}\mathcal{O}_d(M)&=\#\{(m_1,\cdots,m_n)\in \mathbb{N}^n\mid \sum\limits_{i=1}^n m_id^{(i)}_{\min}\leq d\}\\
&\leq \operatorname{dim}_{\mathbb{C}}\mathcal{O}_d(\mathbb{C}^n)-\#\{(m_1,\cdots,m_{n-1})\in \mathbb{N}^{n-1}\mid \sum\limits_{i=1}^n m_i< d\}\\
&=\operatorname{dim}_{\mathbb{C}}\mathcal{O}_d(\mathbb{C}^n)-\binom{n+d-2}{d-1}.
    \end{aligned}
    $$
    with equality holds if and only if $d^{(i)}_{\operatorname{min}}=1$ for $i\leq n-1$ and $1<d^{(n)}_{\operatorname{min}}<\frac{d}{d-1}$. This contradicts with our assumption.

    Now suppose the universal cover $\widetilde{M}$ splits. By the proof of Corollary 7.1 in \cite{Liu16a}, $\widetilde{M}=M_1^{m_1}\times M_2^{n-m_1}$, where $M_1,M_2$ are both simply connected and has nonnegative bisectional curvature. Also $M_1$ is of maximal volume growth. $\mathcal{O}_d(M)$ can be seen as the $G$-invariant subring of $\mathcal{O}_d(M_1)$, $G$ is a compact Lie group acting on $M_1$.

    Therefore 
$$
\begin{aligned}
    \operatorname{dim}_{\mathbb{C}}\mathcal{O}_d(\mathbb{C}^n)-\binom{n+d-2}{d-1}<\operatorname{dim}_{\mathbb{C}}\mathcal{O}_d(M)\leq \operatorname{dim}_{\mathbb{C}}\mathcal{O}_d(M_1),
\end{aligned}
$$
    From the paragraph above, $M_1$ must be biholomorphically isometric to $\mathbb{C}^n$, i.e. $\widetilde{M}=\mathbb{C}^n$. And the $G$-action on $\mathbb{C}^n$ must be the deck transformation of $\pi_1(M)$.

\begin{claim}\label{Claim1}
Let $\pi:\widetilde{M}\rightarrow M$ be the covering map, $f\in \mathcal{O}_P(M)$, then $\operatorname{deg}(f)=\operatorname{deg}(f\circ \pi)$.
\end{claim}
\textbf{Proof of Claim \ref{Claim1}.} It's obvious since $\pi(B(\tilde{p},r))=B(p,r)$ for any $\pi(\tilde{p})=p$ and any $r>0$.
\qed

Since $\mathcal{O}_P(M)$ is finitely generated, take a family of minimal generators of it, say $\{f_1\circ \pi,\cdots,f_N\circ \pi\}\subset \mathcal{O}_P(\mathbb{C}^n)$. By the claim we know that $\operatorname{dim}_{\mathbb{C}}\mathcal{O}_d(M)=\#\{\text{monomial of } f_1\circ \pi,\cdots,f_N\circ\pi \text{ which is in } \mathcal{O}_d(\mathbb{C}^n)\}$.

Similar argument as above shows that 
$$
\operatorname{dim}_{\mathbb{C}}\mathcal{O}_d(M)\leq \operatorname{dim}_{\mathbb{C}}\mathcal{O}_d(\mathbb{C}^n)-\binom{n+d-2}{d-1}
$$
unless $\{f_1\circ \pi,\cdots,f_N\circ \pi\}=A\{z_1,\cdots,z_n\}+b$ for some constant $A,b$, in which case $\pi_1(M)$ is trivial. Therefore we complete the proof.
\end{proof}

\begin{remark}
 Similar argument can also prove the splitting theorem of Section 3 in \cite{CH25}, by using Theorem 4.1 in \cite{NT03}.
\end{remark}

\subsection{Expanding gradient K\"ahler-Ricci solitons with nonnegative Ricci curvature} 
\begin{definition}
    A K\"ahler-Ricci soliton consists of a triple $(M,g,X)$, where $M$ is a K\"ahler manifold, $X$ is a holomorphic vector field on $M$, and $g$ is a complete K\"ahler metric on $M$ whose K\"ahler form $\omega$ satisfies
    $$
    \operatorname{Ric}-\frac{1}{2}\mathcal{L}_X\omega+\lambda\omega=0
    $$
    for some $\lambda\in \{-1,0,1\}$. A K\"ahler-Ricci soliton is said to be expanding if $\lambda =1$. Here $X$ is called the soliton vector field. If in addition, $X=\nabla^g f$ for some real-valued smooth function $f$ on $M$, then we say $(M,g,X)$ is a gradient K\"ahler-Ricci soliton. In this case, we call $f$ the potential function of the soliton. It's equivalent to say
    $$
    \operatorname{Ric}_{i\bar{j}}-f_{i\bar{j}}+\lambda g_{i\bar{j}}=0,\quad f_{ij}=f_{\bar{i}\bar{j}}=0.
    $$
\end{definition}

Then we list some important properties of gradient expanding K\"ahler-Ricci solitons with nonnegative Ricci curvature which can be found in \cite{CT05,Bry08,CCGGIIKLLN,Zhang09,CT11,Yang22}.
\begin{proposition}\label{GEKRS}
    Let $(N^n,J,g,f)$ be a complete noncompact expanding gradient K\"ahler-Ricci soliton with nonnegative Ricci curvature. Then
    
    (1) $R+|\nabla^gf|^2-f$ is constant on $N$, therefore we always normalize $f$ so that $f=R+|\nabla^gf|^2$.

    (2) The potential $f$ is a strictly convex exhaustion function with $\nabla^g f$ has the unique zero at $O\in N$, and $R$ attains its maximum at $O$. In particular, if $N$ has nonnegative bisectional curvature, then $N$ has bounded curvature.

    (3) $(N,g)$ is biholomorphic to $\mathbb{C}^n$ and has maximal volume growth.

    (4) $\nabla^g f$ is a complete vector field. Let $\varphi(t)$ be a family of biholomorphism of $(N,J)$ with
    \begin{equation}\label{selfsimilar}
    \frac{\partial}{\partial t}\varphi(t-1)(x)=-\frac{1}{t}\nabla^gf(\varphi(t-1)(x)), \varphi(0)=\text{id}.
    \end{equation}
    It follows that 
    \begin{equation}\label{solitoneq}
    g(t)=t\varphi(t-1)^* g
    \end{equation}
    solves the K\"ahler-Ricci flow equation on $[0,+\infty)$ with $g(1)=g$.

    (5) ( (13) in \cite{Yang22} ) The potential function $f$ satisfies the estimate
    \begin{equation}\label{potential}
        f(\varphi(O,t))\leq f(O)e^{2t}. 
    \end{equation}

    (6) There exists a global coordinate called the Poincar\'e coordinate on $N$ which is constructed as following: there exists a local coordinate $\left(U_O, z_i\right)$ at $O$ such that $\left\{\frac{\partial}{\partial z_i}\right\}$ is unitary at $O$ and the complexified vector field $\nabla f-\sqrt{-1}J\nabla f=\sum\limits_{i=1}^n\left(1+\mu_i\right) z_i \frac{\partial}{\partial z_i}$. Then one can extend $\left(U_O, z_i\right)$ to the whole $(N, g)$ by defining

$$
z_i(\psi(p, t))=e^{\left(1+\mu_i\right) t} z_i(p)
$$
for any $p \in U_O$ where $\psi(p, t)$ is the holomorphic flow generated by $\nabla f$ starting from $p$.
\end{proposition}

\section{Proof of Theorem \ref{thm:B}}\label{sec:3}
The key observation is the following which may be well-known for experts in the field. However we include the short proof for the reader's convenience.
\begin{lemma}\label{smoothout}
 Let $(N,J,h,O,f)$ be a complete noncompact expanding gradient K\"ahler-Ricci soliton with nonnegative bisectional curvature and we normalize it by assuming $f=R+|\nabla^{h(1)}f|^2$. Let
 $$
 h(t)=t\varphi(t-1)^* h
 $$
 be the unique K\"ahler-Ricci flow solution on $[0,+\infty)$ such that $h(1)=h$. Where
 $$
 \frac{\partial}{\partial t}\varphi(t-1)(x)=-\frac{1}{t}\nabla^gf(\varphi(t-1)(x)),\quad \varphi(0)=\text{id}.
 $$
Then $(N,h(1))$ has a unique tangent cone at infinity $(N_\infty,h_\infty,O_\infty)$. Moreover, $(N,h(t),O)\rightarrow (N_\infty,h_\infty,O_\infty)$ as $t\rightarrow 0$ in the Gromov-Hausdorff sense.
\end{lemma}
\begin{proof}
Replacing $(M,g)$ in Theorem \ref{LongKRflow} by $(N,h(1))$, we conclude that there exists a metric K\"ahler cone $N_\infty$ such that $(N,h(t),O)\rightarrow (N_\infty,h_\infty,O_\infty)$ as $t\rightarrow 0$ in the Gromov-Hausdorff sense. This is based on the distance estimate by Simon-Topping \cite{ST21}, which is in general incorrect without $\operatorname{Ric} \geq 0$. So it remains to prove this cone $N_\infty$ is the unique tangent cone of $(N,h(1))$.

\begin{claim}\label{Claim2}
    $\lim\limits_{t\to 0}d_{th(1)}(\varphi(t-1)(O),O)=0$.
\end{claim}
\textbf{Proof of Claim \ref{Claim2}.} For any $s\leq 1$, by (2) and (\ref{potential}) in Proposition \ref{GEKRS},
    $$
    |\nabla^{h(1)}f|^2\left(\varphi(s)(O)\right)\leq R+|f(\varphi(s)(O))|\leq C(1+e^{2})\leq C.
    $$
    Integrating along $\varphi(s)(O)$, we obtain
    $$
    \sqrt{t}d_{h(1)}(\varphi(t-1)(O),O)\leq \sqrt{t}\int_{t}^1\frac{C}{s}ds=-C\sqrt{t}\log(t)\rightarrow 0 \text{ as } t\rightarrow 0.
    $$
\qed

    In fact now we have
    $$
    \begin{aligned}
    (N_\infty,h_\infty,O_\infty)&=\lim\limits_{t\to 0}(N,h(t),O)\\
    &=\lim\limits_{t\to 0}(N,t\varphi(t-1)^* h(1),O)\\
    &=\lim\limits_{t\to 0}(N,th(1),\varphi(t-1)(O))\\
    &=\lim\limits_{t\to 0}(N,th(1),(O)).
    \end{aligned}
    $$
    All the limits above are in the sense of Gromov-Hausdorff convergence. The second-to-last equality holds by Claim \ref{Claim2}:
\end{proof}

Hence we now have a clear geometric framework characterizing the long-time asymptotics of the K\"ahler-Ricci flow on a complete noncompact K\"ahler manifold with nonnegative bisectional curvature and maximal volume growth, along with its relationship to the tangent cones.

We are now in a position to prove Theorem \ref{thm:B}. 

\textbf{Proof of Theorem \ref{thm:B}.}
 
   Following the notation of Theorem \ref{LongKRflow}, for any sequence $t_i\to t$, the rescaled K\"ahler-Ricci flow $g_i(t)$ converges subsequenctially to a K\"ahler-Ricci flow $(N,J_{\infty},h(t))$ coming out of a tangent cone $M_{\infty}$ of $M$, and $h(t)$ is a family of expanding gradient K\"ahler-Ricci soliton metrics on $N$. Take the Poincar\'e coordinate $(w_1,\cdots,w_n)$ from Proposition \ref{GEKRS} on $(N,h(1),O)$ and normalize them such that
   $$
   \dashint_{B_{h(1)}(O,1)}{w_i}\overline{w_{j}}=\delta_{ij},\quad w_{i}(O)=0\quad \forall i,j\in \{1,\cdots,n\}.
   $$. 
   
   According to the proof of Theorem 1.4 in \cite{Yang22}, the aforementioned Poincar\'e coordinate $(w_1,\cdots,w_n)$ is a canonical coordinate on $(N,h(1),O)$. Following the procedure of part (5) in Theorem \ref{infinitystrucutre} on $(N,h(1),O)$ and $(M,g,p)$, we get two canonical coordinates $(z_{1\infty},\cdots,z_{n\infty})$ and $(w_{1\infty},\cdots,w_{n\infty})$ on $N_{\infty}=M_\infty$, with degrees $\overrightarrow{d_{\operatorname{min}}}(M)$ and $\overrightarrow{d_{\operatorname{min}}}(N)$ by part (3) of Theorem \ref{infinitystrucutre}. Also we know that $\overrightarrow{d_{\operatorname{min}}}(N)=\{\mu_1+1,\cdots,\mu_n+1\}$ by Theorem \ref{Consoli}. By part (4) of Theorem \ref{infinitystrucutre},
   $$
   \overrightarrow{d_{\operatorname{min}}}(M)=\overrightarrow{d_{\operatorname{min}}}(N)=\{\mu_1+1,\cdots,\mu_n+1\}.
   $$
   Therefore, (1), (2) has been proved. For (3), note that in \cite{Liu24}, Liu has shown that $\operatorname{ASCD}(M,g(t))=4n(\sum\limits_{i=1}^n d^{(i)}_{\operatorname{min}}-n)=4n(\sum\limits_{i=1}^n \mu_i)$ is invariant along $t$, based on Simon-Topping's distance estimate. The proof is complete.
   \qed

As an application, we can use it to prove Corollary \ref{bridge}:

\textbf{Proof of Corollary \ref{bridge}}:

    Take a canonical coordinate $(f_1,\cdots,f_n)$ on $(M,g)$, with degree $\overrightarrow{d_{\operatorname{min}}}$. By comparing the distance function $d_{g(t)}(x,p)$ along the flow using the shrinking balls lemma by Simon and Topping (Corollary 3.3 in \cite{ST22}), we can conclude that $(f_1,\cdots,f_n)$ is still a global coordinate on $(M,g(t))$, with the same degree $\overrightarrow{d_{\operatorname{min}}}$. Let $g_k(t)=\frac{1}{t_k}g(t_kt)$ with $t_k\rightarrow +\infty$, we normalize these coordinate functions such that
    $$
    \dashint_{B_{g_k(1)}(p,1)}{f_{ik}}\overline{f_{jk}}=\delta_{ij},\quad f_{ik}(p)=0\quad \forall i,j\in \{1,\cdots,n\}, k\in \mathbb{N}.
    $$
    which makes $(f_{1k},\cdots, f_{nk})$ serving as a canonical coordinate on $(M,g_k(1))$.

    Following a similar procedure as part (5) in Theorem \ref{infinitystrucutre}, $(f_{1k},\cdots, f_{nk})$ converges smoothly to $(f_{1\infty},\cdots,f_{n\infty})$ on $(N,h(1))$ and serves as a canonical coordinate on $N$. Due to part (3) of Theorem \ref{Liuthm1}, we can assume $(f_{1\infty},\cdots,f_{n\infty})$ is just a Poincar\'e coordinate on $(N,h(1))$. Therefore
    $$
    \operatorname{Ric}_{h(1)}\left(\frac{\frac{\partial}{\partial f_{i\infty}}}{\|\frac{\partial}{\partial f_{i\infty}}\|},\overline{\frac{\frac{\partial}{\partial f_{i\infty}}}{\|\frac{\partial}{\partial f_{i\infty}}\|}}\right)=\mu_i.
    $$
    Finally, by the smooth convergence of $g_k(1)$ and the proof of Theorem 1.2 in \cite{CT06}, the Corollary is proved. In fact, for any $i$ and any nonzero $v\in V_1\oplus\cdots\oplus V_i$ but $v\notin V_1\oplus\cdots\oplus V_{i-1}$,
    $$
    \lim\limits_{t\to\infty}t\operatorname{Ric}(v(t),\overline{v(t)})=\mu_i
    $$
    always holds.
\qed
\section{Proof of Theorem \ref{thm:A}}\label{sec:4}
\textbf{Proof of Theorem \ref{thm:A}}:

    Firstly we assume that $M$ is of nonmaximal volume growth, if there exists a nonzero polynomial growth holomorphic $n$-form $s\in H^0_P(M,\mathcal{K}_M)$, the Poincar\'e-Lelong equation states that
    $$
    \frac{\sqrt{-1}}{2\pi}\partial\bar{\partial}\log \|s\|^2=(s)+\operatorname{Ric}\geq 0,
    $$
    where $(s)$ is the zero divisor of $s$. By direct computation, $h:=\log (\|s\|^2+1)$ is a smooth plurisubharmonic function of logarithmic growth. Since the universal cover of $M$ does not split, the heat-flow method of Ni-Tam \cite{NT03} can be applied to perturb $h$ to $h^{\prime}$ which is a smooth strictly plurisubharmonic function of still logarithmic growth. Then we solve the $\bar{\partial}$-equation by applying the H\"ormander $L^2$-estimate to obtain a nonconstant polynomial growth holomorphic function on $M$, contradicting Theorem \ref{Niconj}. As a result, $D_{\operatorname{min}}=+\infty=\sum\limits_{i=1}^n d_{\operatorname{min}}^{(i)}-n$.

    Now we assume that $M$ is of maximal volume growth. Let $(z_1,\cdots,z_n)$ be a canonical coordinate of $M$ with degrees $\left(d_{\operatorname{min}}^{(1)}, \cdots, d_{\operatorname{min}}^{(n)}\right)$. By gradient estimate \cite{CY75},
    $$\label{deg}
    D_{\operatorname{min}}=\operatorname{deg}(dz_1\wedge\cdots\wedge dz_n)\leq \sum\limits_{i=1}^n d_{\operatorname{min}}^{(i)}-n, 
    $$

In the opposite direction, taking a blow-down sequence $(M,g_i=r_i^{-2}g)$ that converges to a tangent cone $M_\infty$ in the Gromov-Hausdorff sense, where $r_i\rightarrow \infty$. 

Assume 
$$
D_{\operatorname{min}}\leq \sum\limits_{i=1}^nd_{\operatorname{min}}^{(i)}-n-\varepsilon
$$
for some $\varepsilon>0$.

Following the notations in part (5) of Theorem \ref{infinitystrucutre}, and define
$$
\varphi_i=\log \|dz_{1i}\wedge\cdots\wedge dz_{ni}\|^2_{g_i}.
$$
As in \cite{Liu24}, since $\varphi_i$ is plurisubharmonic on $M_i$, by passing to a subsequence, we may assume $\varphi_i$ converges to a plurisubharmonic function on $M_\infty$ in $L_{\operatorname{loc}}^1$, denoted by $\varphi=\log |dz|^2$. Therefore $e^{\varphi_i}\rightarrow e^{\varphi}$ almost everywhere.

By the three circle theorem as in the proof of Theorem \ref{infinitystrucutre} (5), there exists a constant $C$ such that
$$
e^{\varphi_i}\leq C r_i^{2D_{\operatorname{min}}}\leq Cr_i^{2(\sum\limits_{i=1}^nd_{\operatorname{min}}^{(i)}-n-\varepsilon)}.
$$
uniformly. Thus $e^{\varphi}\leq Cr^{2(\sum\limits_{i=1}^nd_{\operatorname{min}}^{(i)}-n-\varepsilon)}$ almost everywhere, which contradicts Page 6, Claim 1 in \cite{Liu24} if $\varepsilon>0$, since $e^{\varphi}$ is homogeneous of degree $2(\sum\limits_{i=1}^nd_{\operatorname{min}}^{(i)}-n)$. Therefore $D_{\operatorname{min}}=\sum\limits_{i=1}^n d_{\operatorname{min}}^{(i)}-n.$
\qed
\begin{remark}
    The strategy in the proof of Corollary 1.5 in \cite{Yang22} may not apply in our setting. Although the distance estimate remains valid along the K\"ahler–Ricci flow, the norm of the volume form $\|dz_1\wedge\cdots\wedge dz_n\|_{g(t)}$ cannot be effectively controlled; see (\ref{volformchange}).
\end{remark}
\section{Further discussion}\label{sec:5}

There are several remarks and questions that seem to be interesting to the author.

1. In \cite{HS16} He-Sun established a classification theorem of compact Sasaki manifolds with positive transverse bisectional curvature. We begin by reviewing their work.

Equip the sphere $S^{2n-1}$ with its standard Sasaki structure, whose K\"ahler cone is $\mathbb{C}^n$ with the flat metric $\omega=\frac{\sqrt{-1}}{2}\sum\limits_{i=1}^n dz^i\wedge d\bar{z}^i$. The Reeb field $\xi=\operatorname{Re}(\sqrt{-1}\sum\limits_{i=1}^n z^i\frac{\partial}{\partial z^i})$, the contact form on $S^{2n-1}$ is $\eta=\frac{1}{\sum\limits_{i=1}^n|z^i|^2}\sum\limits_{i=1}^n(y^idx^i-x^idy^i)$. The automorphism group of this Sasaki metric is $U(n)$, take a maximal torus $\mathbb{T}^n$ in $U(n)$ consisting of diagonal elements, with Lie algebra $\mathfrak{t}\cong \mathbb{R}^n$. It is not hard to check that the Reeb cone $\mathcal{R}$ in this case is $\mathbb{R}^n_+$.

We call a Sasaki structure on $\mathbb{S}^n$ simple if it is isomorphic to a Sasaki structure that comes out of a simple deformation from the standard Sasaki structure above. And the corresponding Sasaki manifold is called a weighted Sasaki sphere. For a precise description of  simple deformation, see Section 2 of \cite{HS16}.

Their main result is the following structure theorem.
\begin{theorem}[Theorem 1.1 in \cite{HS16}]
    Let $(M^{2n-1},g)$ be a compact simply connected Sasaki manifold with positive bisectional curvature, then its K\"ahler cone is biholomorphic to $\mathbb{C}^n$. Moreover $M$ is a weighted Sasaki sphere.
\end{theorem}

In our case, we can give a singular version of He-Sun's result. Let $(M,g)$
be a complete noncompact K\"ahler manifold with nonnegative bisectional curvature and maximal volume growth. Let $(M_{\infty}=C(X)\cong \mathbb{C}^n,\omega=\frac{\sqrt{-1}}{2}\partial\bar{\partial}r^2)$ be a tangent cone with Reeb field $\xi=\operatorname{Re}(\sqrt{-1}\sum\limits_{i=1}^n d_{\operatorname{min}}^{(i)}z^i\frac{\partial}{\partial z_i})$, where $(z_1,\cdots,z_n)$ is a canonical coordinate. Choose a smooth K\"ahler cone metric $\hat{\omega}$ on $\mathbb{C}^n$ with same Reeb field $\xi$, whose link is a weighted Sasaki sphere, let the radial function be $\hat{r}$.

It is straightforward to verify that the quotient function $\varphi=\frac{r}{\hat{r}}$ is basic ( i.e. $L_\xi\varphi=L_{J\xi}\varphi=0$ ). According to Proposition 1 in \cite{Liu24}, the current $\sqrt{-1}\partial\bar{\partial}\log r=\sqrt{-1}\partial\bar{\partial}\log (\hat{r}\exp\varphi)={\hat{\omega}}^T+\frac{1}{2}d_Bd_B^c\varphi\geq 0$, which means that $\varphi\in \operatorname{PSH}(S, {\hat{\omega}}^T)$, see \cite{Van15}. Therefore, the link $X$ is a ``singular'' weighted Sasaki sphere.

2. In \cite{Der16} the author showed that a positively curved smooth metric cone can be desingularized by an expanding gradient Ricci soliton with positive curvature operator. Motivated by this, we would expect that the following analogue holds.
\begin{conjecture}
    For any smooth K\"ahler cone metric $(\mathbb{C}^n,\omega=\frac{\sqrt{-1}}{2}\partial\bar{\partial}r^2)$ with nonnegative bisectional curvature, there exists an expanding gradient K\"ahler-Ricci soliton with nonnegative bisectional curvature coming out of this cone.
\end{conjecture}

By Theorem \ref{infinitystrucutre} (2), it is equivalent to ask that whether such cone can be realized as a tangent cone of a K\"ahler manifold with nonnegative bisectional curvature. When the K\"ahler cone is regular, and has positive curvature operators on (1,1)-forms, the conjecture has been solved by Conlon-Deruelle in \cite{CD20} by solving a smooth family of complex Monge-Amp\'ere equations using continuity method. For more related work, we refer to \cite{CDS24}.

\subsection*{Acknowledgements} The author would like to express his sincere gratitude to his advisor, Yihu Yang, for the constant support and lots of inspiring discussions. He would like to extend his thanks to Gang Liu and Zhiqin Lu for many valuable discussions. He is thankful to Man-Chun Lee for bringing the preprint \cite{Lee25} to his attention, as well as for interest in this work and valuable comments. He is grateful to University of California, Irvine for hospitality during his visit in fall 2025. He would like to thank the referee for various suggestions and comments that greatly improve the exposition of this paper and for pointing out some inaccuracy of the previous version. The author is partially supported by NSF of China (No.12071283).
\bibliographystyle{amsalpha}
\bibliography{references}
\end{document}